\documentclass[12pt,a4paper]{amsart}

\usepackage[letterpaper,top=2cm,bottom=2cm,left=3cm,right=3cm,marginparwidth=1.75cm]{geometry}

\usepackage{amsfonts}
\usepackage{mathtools}
\usepackage{graphicx}
\usepackage[colorlinks=true, allcolors=blue]{hyperref}

\usepackage{environments}

\newcommand{\prefix}[1]{\left|#1^{\#} \cap [0, x]\right|}

\author{Karim F. Shamazov}
\address{\parbox{\linewidth}{HSE University, Faculty of Computer Science, \\ 11 Pokrovsky Blvd., 109028 Moscow, Russia}}
\email[]{kfshamazov@edu.hse.ru}

\author{Alexey L. Talambutsa}
\address{\parbox{\linewidth}{Steklov Mathematical Institute of RAS, 8 Gubkina St., 119991 Moscow, Russia\\
HSE University, Laboratory of Theoretical Computer Science, \\ 11 Pokrovsky Blvd., 109028 Moscow, Russia}}
\email[]{altal@mi-ras.ru}

\title[On orbit sets generated by semigroups of affine functions]{On orbit sets generated by semigroups of \\ one-dimensional affine functions}

\begin{document}

\begin{abstract}
The one-dimensional orbit set $\langle F : s \rangle$ is formed by the images of a number $s$ under the action of a semigroup generated by integer affine functions $f_i=a_i x+b_i$ taken from the set $F=\{f_1,\ldots,f_n\}$. P.Erd\H{o}s established an upper bound $O(x^{\sigma+\epsilon})$ for the growth function $|\langle F : s \rangle\cap[0,x]|$, where $1/a_1^{\sigma}+1/a_2^{\sigma}+\ldots + 1/a_n^{\sigma}=1$ and $\varepsilon>0$, which was extended to orbit multisets and real affine functions by J.Lagarias. We complement this by a lower bound $\Omega(x^{\sigma})$ for the multiset size $|\langle F : s \rangle^\#\cap[0,x]|$.

P.~Erd\H{o}s and R.~Graham asked whether an orbit set $\langle F : s \rangle$ has positive density when $F$ is a basis of a free semigroup and $1/a_1+1/a_2+\ldots + 1/a_n=1$. Under these two conditions, we establish a sublinear lower bound $|\langle F : s \rangle \cap [0,x]|=\Omega(x/\log^{\frac{n-1}2} x)$. We also show that in the case when the functions of $F$ form an exact covering system of integers, i.e. when $f_1(\mathbb Z) \sqcup \ldots \sqcup f_n(\mathbb Z)=\mathbb Z$, this bound can be strengthened to $\Omega(x)$, so the set $\langle F : s \rangle$ has positive density.
\end{abstract}

\maketitle

\section{Introduction}

The consideration of one-dimensional integer orbit sets was originated in early 1970's by D.J.~Crampin, A.J.W.~Hilton in connection with their work on constructing self-orthogonal Latin squares of large even sizes. Shortly after that, the general study of such sets was initiated by D.~Klarner and R.~Rado in \cite{KlarnerRado74}. They defined the \emph{orbit set} $\langle F : S\rangle$ as the minimal set containing a given set of integers $S$, which is closed under a given set of operations $F=\{f_1,f_2,\ldots, f_n\}$, each having form $f=m_1x_1+m_2x_2+\ldots+m_dx_d+b$, where the coefficients $m_1,\ldots,m_d,b$ are non-negative integers. 

A natural question of arithmetical combinatorics can be asked for the set $\langle F : S\rangle$: Does it contain an infinite arithmetic progression? It is not known if this problem is algorithmically decidable even in the one-dimensional case $d=1$. A closely related question asks about posivity of the upper asymptotic density 
\[\overline d\left(\langle F : S\rangle\right) = \liminf_{x\to\infty}\frac1x |\langle F : S\rangle \cap [0,x]|.\]

P.Erd\H{o}s obtained upper bound on the size of the orbit sets $\langle F : S \rangle$ in the case when $F=\{f_1,\ldots,f_n\}$, where $f_i(x)=a_ix+b_i$. From this bound it follows, that if $S$ is a finite set and $1/a_1+\ldots+1/a_n<1$, then $\langle F : S\rangle$ has zero upper density, hence it does not contain arithmetic progressions. As a partial problem, for which this method was not directly applicable, Erd\H{o}s suggested to study the density of the set $\langle 2x+1,3x+1,6x+1 \rangle$. Crampin and Hilton showed that for this set the density is zero, extending the method of Erd\H{o}s to the case when $1/a_1+1/a_2+\ldots+1/a_n=1$ and, additionally, there is a non-trivial semigroup relation between the functions from $F$ of the form 
\begin{equation}
f_{u_1}\circ \ldots\circ f_{u_p} = f_{v_1}\circ \ldots\circ f_{v_q},
\label{eq-rel}
\end{equation}
where $u_1,\ldots, u_p,v_1,\ldots, v_q\in \{1,\ldots,k\}$ and $(u_1,\ldots, u_p)\ne (v_1,\ldots, v_s)$. 
After that, Erd\H{o}s and Graham in \cite[p.22--23]{ErdosGraham} posed a general question on positivity of the density, where they highlighted the case when $1/a_1+1/a_2+\ldots + 1/a_n=1$ and the functions of $F$ form a basis of a free semigroup, i.e. no relation having form \eqref{eq-rel} exists.

The results of Erd\H{o}s, Crampin and Hilton brought Klarner to studying the conditions, under which there is no relation between a set of one-dimensional affine functions, and he established some necessary and sufficient conditions for this in \cite{Klarner82}. As a corollary he found six integer triples $(a,b,c)$, such that the functions $f_2=2x+a,f_3=3x+b,f_6=6x+c$ form a basis of a free semigroup and suggested to study if the set $\langle f_2,f_3,f_6 : 0 \rangle$ has positive density. This question remains unanswered for all these triples, including $(0,2,3)$ for which the problem reappeared in the famous paper \cite[Problem 4]{Guy} and monograph \cite[E36]{Guy-monograph} by R.~Guy. 

A nice recent exposition of the results described above and some other related results can be found in the paper \cite{Lagarias} by J.Lagarias in more details. Lagarias suggested to generalize the orbit sets to the functions with non-integer coefficients, which allowed to reformulate Collatz conjecture as a question about an orbit set. He also extended the upper bound of Erd\H{o}s to multisets $\langle F : S\rangle^\#$, where $S$ is a starting multiset and $F=\{f_i \mid i\in \mathcal I\}$ is a (possibly infinite) set of affine functions
\begin{equation*}
f_i(x) = a_i x + b_i, \text{ where } a_i\geq 1 \text{ and } b_i\geq 0.
\end{equation*}
Then $\langle F : S\rangle^\#$ is a union of multisets $\bigcup_{i=0}^\infty S_i$, where $S_0=S$ and $S_{i+1}=\bigcup_{j=1}^kf_j(S_i)$. The \emph{growth function} $|\langle F : S\rangle^\#\cap [0,x]|$ is then well defined once the multiset $S\subset \mathbb R_{>0}$ is discrete (i.e. has no finite limit points). 

Under the conditions above the result of Erd\H{o}s (reported in \cite[Theorem 8]{KlarnerRado74}) was generalized to the following statement (see \cite[Theorem 3]{Lagarias}):

\begin{theorem}[Erdős--Lagarias Upper Bound]
\label{EL-theorem}
Let $F$ be a (possibly infinite) set of affine functions
\begin{equation*}
F = \{f_i(x) = a_i x + b_i \mid i \in \mathcal I\},
\end{equation*}
where $a_i \geq 1, b_i \geq 0$. Suppose that $\sigma$ is a positive real number such that
\begin{equation*}
\alpha = \sum_{i \in \mathcal I} \frac 1{a_i^{\sigma}} < 1.
\end{equation*}
Let $S \subset \mathbb R_+$~ be a discrete set. Then for any $x \geq 1$
\begin{equation*}
\prefix {\langle F : S \rangle} \leq \frac 1{1 - \alpha}\left(\sum_ {\{s \in S \, \mid \, 0 \leq s \leq x\}} \frac 1{s^{\sigma}}\right)x^{\sigma}.
\end{equation*}
\end{theorem}

Clearly, an upper bound for  a multiset $\langle F : S\rangle^\#$ serves as a bound for the set $\langle F : S\rangle^\#$.

\smallskip

The goal of this paper is to prove a lower bound on the size of one-dimensional orbit multisets, which complements the result of Erd\H{o}s-Lagarias, to provide almost linear bounds for the growth function of the orbit sets from the question of Erd\H{o}s-Graham, and to answer the latter question in the special case of exact covering systems.

\section{Lower bound for the image multiset}



In this section we will prove the following lower bound for the size of the multiset $\langle F : S\rangle^\#$, which complements the upper bound of Erd\H{o}s and Lagarias. It will be used in Section \ref{ecc-section} to establish the lower bounds on the sizes of the orbit sets, which are generated by the exact covering systems.

\newcommand{\filterc}{x(\delta - 1) \geq \denominator}
\newcommand{\denominator}{s(\delta - 1) + \beta}

\begin{theorem}\label{lowerbound}
Let $F$ be a finite set of affine functions
\begin{equation*}
F = \{f_i(x) = a_i x + b_i \mid i = 1,\ldots,n\},
\end{equation*}
such that $a_i > 1, b_i \geq 0$. Suppose that $\sigma$ is a positive real number such that
\begin{equation*}
\sum_{i=1}^n \frac 1{a_i^{\sigma}} = 1.
\end{equation*}
Let $S \subset \mathbb R_+$~ be a discrete set. Then for any $x \geq 1$
\begin{equation*}
\prefix {\langle F : S \rangle} \geq \frac {(\delta - 1)^{\sigma}}{\gamma^{\sigma}}\left(\sum_ {\substack{s \in S \\ \filterc}} \frac 1{(\denominator)^{\sigma}}\right)x^{\sigma},
\end{equation*}
where $\delta = \min\limits_{i=1,\ldots,n} a_i, \beta = \max\limits_{i=1,\ldots,n} b_i, \gamma = \max\limits_{i=1,\ldots,n} a_i$.
\end{theorem}

\begin{remark}
The statements of Theorems \ref{EL-theorem} and \ref{lowerbound} seem quite similar, yet together they do not characterize the asymptotic behavior of the function $\prefix {\langle F : S \rangle}$ even when the set $S$ is finite. Indeed, the lower bound of Theorem \ref{lowerbound} is $\Omega(x^{\sigma})$, where $\sum_{i=1}^n \frac 1{a_i^{\sigma}}=1$ and the upper bound that can be provided by Theorem \ref{EL-theorem} is only $O(x^{\sigma+\varepsilon})$ for any $\varepsilon>0$. One sees that there is a gap between these two estimations, which includes, for example, functions $cx^{\sigma}\log^d x$, where $c,d>0$. 
Thus, we are left with the question whether for any function set $F$ as stated and a finite set $S$ one has $\prefix {\langle F : S \rangle}=\Theta(x^{\sigma})$, i.e., whether there exist two constants $c_1,c_2>0$ such that $c_1 x^{\sigma}<\prefix {\langle F : S \rangle}<c_2 x^{\sigma}$.
An example of the set $S_a=\langle 2x, 2x+1 : a\rangle$ from \cite[p.770]{Lagarias} shows that one cannot hope for a general estimate that is more precise than $\Theta(x^{\sigma})$ as the growth function in this example oscillates between $x/a$ and $2x/(a+1)$.
\end{remark}

To establish Theorem \ref{lowerbound} we first prove the following auxiliary result:

\begin{lemma}
Let $F$ be a finite set of affine functions $F = \{f_i(x) = a_i x + b_i \mid i=1,\ldots,n\}$,
such that $a_i > 1, b_i \geq 0$.
Then for any tuple $I=(i_1,\ldots,i_r)\in \{1,\ldots,n\}^r$ and any $s\in \mathbb R$
\begin{equation*}
f_I(s) 
\leq \frac{\denominator}{\delta - 1} a_{i_1}a_{i_2}\ldots a_{i_r},
\label{upper_est}
\end{equation*} 
where $\delta = \min\limits_{i=1,\ldots,n} a_i, \beta = \max\limits_{i=1,\ldots,n} b_i$.
\end{lemma}
\begin{proof}
First, a straightforward check shows that 
\begin{equation}
f_I(s) = a_{i_1}a_{i_2}\ldots a_{i_r}s + \sum_{j = 1}^r a_{i_1}a_{i_2}\ldots a_{i_{j - 1}}b_j.
\label{direct_computation}
\end{equation} 

In order to establish the desired inequality $f_I(s)\leq \dfrac{\denominator}{\delta - 1} a_{i_1}a_{i_2}\ldots a_{i_r}$ we prove the following auxiliary estimation:
\begin{equation}
\sum_{j = 1}^r a_{i_1}a_{i_2}\ldots a_{i_{j - 1}} < \frac{a_{i_1}a_{i_2}\ldots a_{i_r}}{\delta - 1}.
\label{inter_ineq}
\end{equation}
To check this, we consider the following equivalent inequality obtained by multiplication by $\delta-1$ and subsequent rearrangement of the terms.
\begin{equation*}
\sum_{j = 1}^{r - 1} (a_{i_1}a_{i_2}\ldots a_{i_{j - 1}}\delta - a_{i_1}a_{i_2}\ldots a_{i_{j - 1}}a_{i_j}) + a_{i_1}a_{i_2}\ldots a_{i_{r - 1}}\delta - 1 < a_{i_1}a_{i_2}\ldots a_{i_r}.
\end{equation*}
Since $\delta=\min a_i$, the sum consists of non-positive values, so this inequality and \eqref{inter_ineq} hold true. As, moreover, $\beta=\max b_j$, from \eqref{direct_computation} and \eqref{inter_ineq} we obtain the Lemma \ref{upper_est}.
\end{proof}

\begin{proof}[Proof of Theorem \ref{lowerbound}.]
For positive $x \geq 1$, we define
\begin{equation*}
N(x) = \{I = (i_1, \ldots, i_r) \mid r \geq 0, a_{i_1}\ldots a_{i_r} \leq x\}.
\end{equation*}
The size of this set is $|N(x)| \geq 1$ for $1 \leq x < \gamma$. First we will show that for any $x \geq 1$
\begin{equation}
|N(x)| \geq \frac {1}{\gamma^{\sigma}}x^{\sigma} 
\label{Nx_lower_bound}
\end{equation}

We will prove this claim for the half-interval $1 \leq x < \delta^k\gamma$ by induction on $k$. The base for $k = 0$ is:
\begin{equation*}
|N(x)| \geq 1 = \frac{\gamma^{\sigma}}{\gamma^{\sigma}} \geq \frac {1}{\gamma^{\sigma}}x^{\sigma}
\end{equation*}
Inductive step from $k$ to $k + 1$ is:
\begin{equation*}
|N(x)| = 1 + \sum_{i=1}^n\left|N\left(\frac{x}{a_i}\right)\right|.
\end{equation*}
Since $1 \leq \delta^k \leq \frac x{\gamma} \leq \frac x{a_i} \leq \frac x{\delta} < \delta^k\gamma$, we can apply the induction hypothesis:
\begin{eqnarray*}
|N(x)| &\geq& 1 + \frac {1}{\gamma^{\sigma}}\left(\sum_{i=1}^n \left(\frac x{a_i}\right)^{\sigma}\right)\\
&=& 1 + \frac {1}{\gamma^{\sigma}}\left(\sum_{i=1}^n \frac 1{a_i^{\sigma}}\right)x^{\sigma}\\
&=& 1 + \frac {1}{\gamma^{\sigma}}x^{\sigma} \geq \frac {1}{\gamma^{\sigma}}x^{\sigma}.
\end{eqnarray*}

Now we will show that for any $x > 0$
\begin{equation}
\prefix {\langle F : \{s\} \rangle} \geq \left|N\left(\frac{x(\delta - 1)}{\denominator}\right)\right|.
\label{lower_estimation}
\end{equation}
Let us consider such tuples $I=(i_1,\ldots, i_r)$ such that 
\begin{equation}
\dfrac{\denominator}{\delta - 1} a_{i_1}a_{i_2}\ldots a_{i_r} \leq x.
\label{Ix-bound}
\end{equation}
Then, according to Lemma \ref{upper_est} we obtain that for any such tuple $I$ we have $f_I(s) \leq x$. The definition of the set $N$ shows that any tuple $I$ satisfying \eqref{Ix-bound} is in the set 
\[
N\left(\frac{x(\delta - 1)}{\denominator}\right),
\] so the estimation \eqref{lower_estimation} follows.


\medskip

For $\filterc$ the right hand side of \eqref{lower_estimation} can be estimated from below by $\frac{1}{\gamma^{\sigma}}\left(\frac{x(\delta - 1)}{\denominator}\right)^{\sigma}$ due to bound \eqref{Nx_lower_bound}. For others $\left|N\left(\frac{x(\delta - 1)}{\denominator}\right)\right| = 0$. This allows us to finally obtain the main claim of the Theorem:

\begin{eqnarray*}
\prefix {\langle F : S \rangle} &=& \sum_{\substack{s\in S}} \prefix {\langle F : \{s\} \rangle}\\
&\geq& \frac{1}{\gamma^{\sigma}}\left(\sum_{\substack{s \in S \\ \filterc}}\left(\frac{x(\delta - 1)}{\denominator}\right)^{\sigma}\right)\\
&=& \frac{(\delta - 1)^{\sigma}}{\gamma^{\sigma}}\left(\sum_{\substack{s \in S \\ \filterc}}\frac{1}{(\denominator)^{\sigma}}\right)x^{\sigma}.
\end{eqnarray*}
\end{proof}


\section{Sublinear bound for Erd\H{o}s-Graham question}

\begin{theorem}\label{subbound}
Let $F$ be a finite set of affine functions
\begin{equation*}
F = \{f_i(x) = a_i x + b_i \mid i=1,\ldots, n\},
\end{equation*}
such that $a_i > 1, b_i \geq 0$. Suppose that $F$ is a basis of a free semigroup and $\sum\limits_{i=1}^n\frac1{a_i}=1$. Then for any $s\geq 0$
\begin{equation*}
\left|\langle F : s \rangle\cap [0, x]\right| = \Omega(x/\log^{\frac{n-1}2}(x)), \text{ when }x\to \infty.
\end{equation*}
\end{theorem}

The main estimate follows from the next Lemma, which shows that we can construct sufficiently many compositions having equal slope and distinct free coefficients.


\begin{lemma}
\label{density_estimation}
Let $F$ be a set of affine functions $F = \{f_i(x) = a_i x + b_i \mid i=1,\ldots,n\}$ such that $a_i > 1, b_i \geq 0$. Suppose that $F$ is a basis of a free semigroup. Then for any $k_1,\ldots,k_n\in \mathbb N$ there exists a set $\mathcal C$ consisting of compositions of functions from $F$ having equal slope $\prod\limits_{i=1}^n a_i^{k_i}$, such that 

a) $|\mathcal S|\geq\left(\sum\limits_{i=1}^n k_i\right)!/(\prod\limits_{i=1}^n k_i!)$

b) For any function $f\in C$ and any $s\geq 0$ the value $f(s)$ is bounded as
\[ 0\leq f(s)\leq M(k_1,k_2,\ldots, k_n;s)=\left(\dfrac{\max_{i=1}^n b_i}{\min_{i=1}^n a_i - 1}+s\right) \prod\limits_{i=1}^n a_i^{k_i}.\]
\end{lemma}
\begin{proof}
Let $\mathcal C$ be the set of compositions $f_{i_1}\circ f_{i_2} \circ \ldots \circ f_{i_L}$, where $L=k_1+k_2+\ldots + k_n$ and the tuple $(i_1,i_2,\ldots,i_L)$ contains each number $i\in \{1,\ldots, n\}$ exactly $k_i$ times. In this case, the slope of any function $f\in \mathcal C$ is equal to $\prod\limits_{i=1}^n a_i^{k_i}$, and all functions in $\mathcal  C$ are distinct because $F$ is a basis of a free semigroup. The number of tuples $(i_1,i_2,\ldots,i_d)$ with the property described above is then equal to the number of words in the alphabet $\{1,\ldots,n\}$ containing $k_1,k_2,\ldots,k_n$ letters of the respective type, which is $\left(\sum\limits_{i=1}^n k_i\right)!/(\prod\limits_{i=1}^n k_i!)$ due to the multinomial theorem.

In (b) the estimation for the value $f(s)$ is immediate with the use of Lemma \ref{upper_est}.
\end{proof}



\begin{proof}[Proof of Theorem \ref{subbound}]
For any argument $t\geq 0$ we define functions $k_i(t)=\left\lfloor\dfrac{t}{a_i}\right\rfloor$, hence
\begin{equation}
k_i(t)=\dfrac{t}{a_i}+O(1).
\label{ki-estimation}
\end{equation} which are non-negative integers. Summing all $k_i$ we obtain
\begin{equation}
\sum\limits_{i=1}^n k_i =\sum\limits_{i=1}^n \left(\frac{t}{a_i} + O(1)\right) = t \left(\sum\limits_{i=1}^n \frac1{a_i}\right) + O(1) = t + O(1).
\label{sum-k}
\end{equation}

Now we can use Lemma \ref{density_estimation} for integer parameters $k_i$ being the described functions of integer $t$ to construct the corresponding set $\mathcal C(t)$. Evaluating the functions from $\mathcal C(t)$ in point $s$ we obtain integer numbers from $\langle F : s 
\rangle$ belonging to the segment $[0,M(k_1(t),k_2(t),\ldots, k_n(t);s)]$, and their number is bounded from below by \[
N=\dfrac{\left(\sum\limits_{i=1}^n k_i\right)!}{\left(\prod\limits_{i=1}^n k_i!\right)}
.\]

In what follows further we will be establishing the bound for $N$ which is suitable to prove the main statement of Theorem. We consider the expression above as a function of $t$ and take a logarithm of it, and using a simplification of Stirling's formula $\log(x!)=x\log x-x+\frac{1}{2}\log 2\pi x + O(1/x)=x\log x-x+\frac{1}{2}\log x + O(1)$ for $x\to \infty$, we obtain
\begin{multline*}
\log N = (\sum\limits_{i=1}^n k_i)\log \sum\limits_{i=1}^n k_i-\sum\limits_{i=1}^n k_i+\frac12\log\sum\limits_{i=1}^n k_i + O\left(1\right) \\
- \sum\limits_{i=1}^n \left(k_i \log k_i - k_i + \frac12\log k_i + O\left(1\right)\right).
\end{multline*}

Canceling out $\sum\limits_{i=1}^n k_i$ and grouping the other terms we rewrite the latter as 
\begin{equation}
\log N=\sum\limits_{i=1}^n \left(k_i \log \sum\limits_{j=1}^n k_j-k_i \log k_i\right)
+ \frac{1}{2}\left(\log\sum\limits_{i=1}^n k_i - \sum\limits_{i=1}^n \log k_i\right) + O(1).
\label{long_expression}
\end{equation}

Using \eqref{ki-estimation} and \eqref{sum-k} together with $\log (1 + x) = x+o(x)$, we obtain
\begin{equation}
\log k_i =
\log t-\log a_i + O\left(1/t\right)
= \log t + O(1),
\label{log-k}
\end{equation}
\begin{equation}
\log \sum\limits_{j=1}^n k_j = \log t+
O\left(1/t\right)
= \log t + O(1),
\label{log-sum}
\end{equation}
We use two equalities of \eqref{log-sum} to rewrite \eqref{long_expression} as
\begin{equation}
\log N=\sum\limits_{i=1}^n\left (k_i (\log t+
O\left(1/t\right)
)-k_i \log k_i
\right)
+ \frac{1}{2}\left(\log t- \sum\limits_{i=1}^n \log k_i\right) + O(1).
\label{two-summands}
\end{equation}


Using \eqref{log-k} we further rewrite the first sum in \eqref{two-summands} as 
\[
\sum\limits_{i=1}^n \left(k_i \left(\log t+
O\left(1/t\right)
-(\log t-\log a_i+
O\left(1/t\right)
)
\right)\right)= \sum\limits_{i=1}^n \left(k_i \left(\log a_i + O\left(1/t\right)\right)\right),
\]
and then applying \eqref{log-k} to $\sum\limits_{i=1}^n \log k_i$ in the second part of  \eqref{two-summands}, we come up with 
\begin{equation}
\log N = \sum\limits_{i=1}^n\left(k_i \left(\log a_i + O\left(1/t\right)\right)\right) - \left(\frac{n-1}2\right)\log t + O(1).
\label{log-expression}
\end{equation}

Using \eqref{ki-estimation} we can rewrite \eqref{log-expression} as
\begin{equation}
\log N = t
\sum\limits_{i=1}^n \frac{\log a_i}{a_i} - \left(\frac{n-1}2\right)\log t + O(1).
\label{final-expression}
\end{equation}


For any fixed $s\geq 0$ we estimate $M(t)=M(k_1(t),\ldots, k_n(t);s)$ as
\[
M(t)=\left(\dfrac{\max_{i=1}^n b_i}{\min_{i=1}^n a_i - 1}+s\right) \prod\limits_{i=1}^n a_i^{\lfloor t/a_i\rfloor} = C\cdot\prod\limits_{i=1}^n a_i^{\lfloor t/a_i\rfloor}
 \leq C\cdot\prod\limits_{i=1}^n a_i^{ t/a_i}.
\]
Then for any $x\geq 0$ we select the following $t$ such that $M(t)\leq x$

\begin{equation*}
t = \frac{\log x - \log C}{
\sum\limits_{i=1}^n \frac{\log a_i}{a_i}},
\end{equation*}
which is positive for $x$ being large enough, and clearly
\begin{equation}
t = \frac{\log x}{
\sum\limits_{i=1}^n \frac{\log a_i}{a_i}} + O(1).
\label{t-expression}
\end{equation}

Now we put \eqref{t-expression} into \eqref{final-expression} and obtain
\begin{equation*}
\log N = \log x - \left(\frac{n-1}2\right)\log \log x + O(1).
\end{equation*}
The exponentiation of this equation gives the estimation 
\begin{equation*}
N = \left(x / \log^{\frac{n-1}2} x\right) e^{O(1)} = \Theta\left(x / \log^{\frac{n-1}2} x\right),
\end{equation*}
which is sufficient to establish the main statement of Theorem.
\end{proof}

\section{Partitions of integers and ping-pong with remainders}

\label{ecc-section}

\subsection{Klarner tuples and the Ping-pong lemma}

In \cite{Klarner82} Klarner established some necessary and sufficient conditions for a set of functions $F=\{ a_ix+b_i\ \mid i=1,\ldots,k\}$ to generate a free semigroup under an operation of composition. The statement of the sufficiency theorem was rather intricate, and the proof involved technical arguments using orders on the sequence of compositions. The paper  \cite{KolTal} suggested a geometrical argument, which both simplified the meaning of the statement and the proof of this theorem. The argument was based on the following version of Ping-pong Lemma, which is a classical tool that can be used to show that a certain set of functions generate a free semigroup under a composition operation. 

\medskip

\textbf{The Ping--Pong Lemma.} \emph{Let $S$ be a semigroup generated by the set of functions $\{f_1, \ldots, f_r\}$ inside the group $\mathrm{Aff}(\mathbb{R})$. If there exists a collection of non-empty  mutually non-intersecting sets $I_1, \ldots, I_r$ such that $f_i(\cup_{j=1}^n I_j) \subset I_i$ for all $1\le i \le r$, then $S$ is a free semigroup of rank $r$ with free basis $f_1, \ldots, f_r$.}

\medskip

This lemma was used in \cite{KolTal} for the action of the inverse functions $f_i^{-1}=x/a_i-b_i/a_i$ on some open interval $I\subset\mathbb R$. As we want to allocate the image intervals $f_1^{-1}(I),\ldots, f^{-1}_r(I)$ inside $I$ without intersection, we have to satisfy the inequality $1/a_1+\ldots+1/a_n\leq 1$.
In the case of equality, such allocation becomes rigid and is defined by the order of the image intervals in $I$. If the intervals go from left to right, then the right end of $f_i^{-1}(I)$ should be exactly the left end of $f_{i+1}^{-1}(I)$. An easy calculation then gives a formula which for a given tuple of positive numbers $(a_1, a_2, \ldots, a_n)$ with $1/a_1+\ldots+1/a_n=1$ computes a tuple $(b_1,b_2,\ldots,b_n)$ such that $F^{-1}$ (and hence $F$ too) generates a free semigroup. 

An explicit formula for the functions from $F$ would be 
\begin{equation}
f_i(x) = a_i x + d_1 a_i \sum_{j < i} \frac 1{a_j} + d_2(a_i - 1),
\label{tuple-generation}
\end{equation}
where two real numbers $d_1$ and $d_2$ are parameters, which correspond to the location of the open interval $I\subset \mathbb R$. Namely, $I=(d_2-d_1,d_2)$, so $d_1$ is the width of the interval $I$ and $d_2$ is the location of its right end. The change of the parameters $d_1$ and $d_2$ corresponds to the conjugation of functions from $F$ by an affine function $g=d_1^{-1}x+d_2/d_1$, i.e. consideration of the set $g^{-1}Fg$. 

If we are given a set of slopes $\{a_1,a_2,\ldots,a_n\}$, then we can order it in at most $k!$ ways, for each of which we can compute the tuple of free coefficients by \eqref{tuple-generation}. The integer tuples of Klarner's theorem for $1/a_1+1/a_2+\ldots 1/a_n=1$ can be generated in such a way if we take $d_2=0$ and choose $d_1\in\mathbb Q$ so that
\[
d_1 a_i \sum_{j < i} \frac{1}{a_j}
\]
are integers and collectively coprime. As, according to Theorem~1 in \cite{KolTal}, this exactly corresponds to the conditions stated originally in Theorem 2.3 of \cite{Klarner82}, we suggest a name of a \emph{Klarner tuple} to any $(b_1,\ldots,b_r)$ computed thus from an integer tuple $(a_1,\ldots,a_r)$.\footnote{The formula \eqref{tuple-generation} can be used not only for integer slopes, but also in the rational and irrational cases, in which the choice of $d_1=1$ and $d_2=0$ provides a neat normal form for the equivalence class.}


\smallskip

For the set of slopes $A=\{2,3,6\}$ the procedure described above generates the same six tuples enlisted in \cite{Klarner82}: 
\begin{equation}
(0,3,10), \quad (0,2,3), \quad (2,0,15), \quad (1,0,2), \quad (2,1,0), \quad (1,6,0).
\label{klarner-tuples}
\end{equation}

If we take $A'=\{2,4,4\}$ we obtain three tuples: 
\begin{equation}
(0,2,3), \quad (1,0,6), \quad (1,0,1).
\label{244-tuples}
\end{equation}
Even though, for slopes $A$ and tuples \eqref{klarner-tuples} the density problem remains unsolved, we will show that for $A'$ of the tuples in \eqref{244-tuples} gives the positive density of $\langle F : s\rangle$.

\subsection{Ping-pong lemma with remainders and exact covering systems}

The Ping-Pong Lemma can also be employed for the direct functions of $F$ for some discrete sets, for example we can take $S=\mathbb Z$. If the coefficients of the functions $f_i=a_ix+b_i$ are integers, we obtain that $f_i(S)$ is an arithmetic progression of numbers equal to $b_i$ mod $a_i$. Thus, if the progressions $f_1(\mathbb Z),\ldots,f_n(\mathbb Z)$ are pairwise non-intersecting, the Ping-Pong Lemma guarantees that the functions $f_1,\ldots,f_n$ generate the free semigroup. Clearly, the non-intersection condition together with the equality $1/a_1+\ldots+1/a_n=1$ brings us to the case, when 
\[f_1(\mathbb Z) \sqcup \ldots \sqcup f_n(\mathbb Z)=\mathbb Z.\] 
The latter partition of integers is known as an \emph{exact covering system} (see \cite[F14]{Guy-monograph}. Interestingly, this topic was also popularized by Erd\H{o}s, and it was described in the monograph of Erd\H{o}s and Graham \cite[p.24--26] {ErdosGraham} right after the question on orbit sets, but apparently no connection between these two topics was mentioned. We will show, that there is one.

\begin{theorem}
If $F=\{ a_ix+b_i\ \mid i=1,\ldots,n\}$ with integer $a_i>1$ and $b_i\geq 0$ provides a system of exact covering congruences, then $\langle F : s \rangle$ has positive density for any $s\geq 0$.
\label{ecc-theorem}
\end{theorem}
\begin{proof}
We will show that the multiplicity of any element in the set $P=\langle F^\# : s\rangle$ is at most one. Indeed, suppose the contrary and since $P\subset \mathbb{Z}_{\geq 0}$, there exists a minimal number $x\in P$ such that its multiplicity is greater than $1$. Then we necessarily have $f_i(u)=f_j(v)$ for some $u,v \in P$ and $u,v\leq x$. However, if $i=j$ then we get a contradiction with the fact that $x$ was minimal, whilst in the case $i\ne j$ we obtain it in $f_i(\mathbb Z)\cap f_j(\mathbb Z)\ne \emptyset$. The multiplicity bound together with Theorem \ref{lowerbound} immediately gives us a linear lower bound $\Omega(x)$ for the growth function $|\langle F : s\rangle|\cap [0,x]$, from which the positivity of $\underline{d}(\langle F : s\rangle)$ follows.
\end{proof}

Now, considering the tuples of \eqref{244-tuples}, we see that the functions $2x+1, 4x+2, 4x+4$ constitute a system of exact covering congruences, so we can apply Theorem \ref{ecc-theorem} for this example of Klarner type  the answer to Erd\H{o}s-Graham question is positive. It is not hard to generate more examples. First, we clearly can change an argument of any progression $ax+b$, substituting it by $ax+b-ah$ for any integer $h$; this gives the same progression but is a different function. Secondly, one can subdivide one progression $ax+b$ to $r$ progressions 
\[
arx+rb_1r+b, \quad arx+rb_2+b, \quad \ldots, \quad arx+rb_r+b, 
\]
where the integers $b_1,\ldots,b_r$ are distinct modulo $r$. This natural operation is called subdivision. S.Porubsky has shown in \cite{Porubsky} that there exist examples, which cannot be obtained by subdivision operation. His smallest example consists of $14$ progressions, but can be shortened to $13$ progressions by considering a projection on the set of even numbers. The corresponding progressions are\footnote{One can see that this set of progressions cannot be generated by the subdivision operation since $\gcd(a_1,\ldots,a_{13})=1$, which is false starting from the first step of any sequence of subdivision operations.}
\begin{eqnarray*}
6x,\quad 6x + 4, \quad 10x + 1, \quad 10x + 3, \quad 10x + 5, \quad 10x + 9, \quad
15x + 2, \\ 30x + 7, \quad 30x + 8, \quad 30x + 14, \quad 30x + 20, \quad 30x + 26, \quad 30x + 27.
\end{eqnarray*}
As this is a system of exact covering congruences, we can also apply Theorem \ref{ecc-theorem} to it.

\smallskip

Unfortunately, for the slope set $\{2,3,6\}$ all cases of free coefficients lead to progressions, which intersect. Moreover, the use of Theorem \ref{ecc-theorem} is impossible when the numbers in the slope set are distinct. This is due to the result obtained in 1950's independently by L.~Mirsky, M.~Newman and B.~Novák--Š.~Znám (see \cite[p.25]{ErdosGraham}, \cite{Newman}, \cite{NovakZnam}), which shows that any system of exact congruences has two progressions with equal differences. A generalization of this result from the group of integers to any group is known as Hertzog-Sch\"onheim conjecture, which is open for more than 50 years (see \cite{HerzogSchoenheim}, \cite{Ginosar}).





\section{Acknowledgments}

The authors are indebted to the anonymous referee for helpful comments on an earlier draft of this paper. The paper was prepared with the support of the Theoretical Physics and Mathematics Advancement Foundation ``BASIS'', grant No. 22-7-2-32-1. The work of the second author was prepared within the framework of the HSE University Basic Research Program.

\bibliographystyle{alpha}

\end{document}